\documentclass[12pt,reqno,letter]{amsart}
\usepackage[utf8]{inputenc}
\usepackage{amsmath,amsthm,verbatim,amscd,amssymb,graphicx,color,float}
\usepackage{setspace}
\usepackage[english]{babel}
\usepackage{enumerate}
\usepackage[plainpages=false]{hyperref}

\newtheorem{thm}{Theorem}[section]
\newtheorem{cor}[thm]{Corollary}
\newtheorem{lem}[thm]{Lemma}

\newtheorem{prop}[thm]{Proposition}

\theoremstyle{definition}
\newtheorem{defi}[thm]{Definition}
\newtheorem{rmk}[thm]{Remark}

\title{Uniform Lipschitz continuity of the isoperimetric profile of compact surfaces under normalized Ricci Flow}

\title[\resizebox{6in}{!}{Uniform Lipschitz continuity of the isoperimetric profile of compact surfaces under normalized Ricci Flow}]{Uniform Lipschitz continuity of the isoperimetric profile of compact surfaces under normalized Ricci Flow}

\author{YIZHONG ZHENG}
\address{Department of Mathematics, The Graduate Center of City University of New York, New York, NY 10016, USA}
\email{yzheng@gradcenter.cuny.edu}

\date{April 2nd, 2020}

\begin{document}

\maketitle

\begin{abstract}
    We show that the isoperimetric profile $h_{g(t)}(\xi)$ of a compact Riemannian manifold $(M,g)$ is jointly continuous when metrics $g(t)$ vary continuously. We also show that, when $M$ is a compact surface and $g(t)$ evolves under normalized Ricci flow, $h^2_{g(t)}(\xi)$ is uniform Lipschitz continuous and hence $h_{g(t)}(\xi)$ is uniform locally Lipschitz continuous.
\end{abstract}

\section{Introduction}
\subsection{Background and main results}

Isoperimetric profile function $h(v)$ of a Riemannian manifold $(M,g)$ of dimension $\geqslant 2$ is the least boundary area enclosing
a given volume $v$. It is well known that the isoperimetric profile is continuous when $M$ is compact, for example, Bavard-Pansu \cite{BP} show that $h$ is locally Lipschitz continuous. Recently, there have been many studies on this function. Nardulli-Russo \cite{NF} show that the continuity also holds for complete manifolds of finite volume. Ritor$\acute{e}$ \cite{Rit} shows that the continuity holds for Hadamard manifolds and complete non-compact manifolds of strictly
positive sectional curvature. However, Nardulli-Pansu \cite{NP} (for dim$\geqslant3$) and  Papasoglu-Swenson \cite{PS} (for dim=2) show that the isoperimetric profile could be discontinuous for some complete connected non-compact Riemannian manifolds. 

In the paper \cite{AB}, Andrews-Bryan prove a comparison theorem for the isoperimetric profiles of solutions of normalized Ricci flow on the two-sphere ($S^2$). They apply the comparison theorem using the Rosenau solution on $S^2$ as the model metric to deduce sharp time-dependent curvature bounds for arbitrary solutions of normalized Ricci flow on $S^2$. Their theorem gives a simple and direct proof of convergence of Ricci flow to a metric of constant curvature on $S^2$. Inspired by Andrews-Bryan's work, we study the isoperimetric profile on compact Riemannian manifolds $(M^n,g)$ ($n\geq{2}$) under continuous variation of metrics and in more depth on compact surfaces under normalized Ricci flow. 

In contrast to the above articles, we choose the volume ratio $\xi \in (0,1)$ as the domain of $h(\xi)$ in this paper rather than the volume. This is because we mainly consider compact manifolds (which have finite volumes) and varying the metric $g$ would potentially change the volume of $(M,g)$. The following are the main results of this paper:

\begin{thm}\label{thm: joint cty of h}
Let $g(t)$ be a family of Riemannian metrics on a compact manifold $M^n$ ($n\geq{2}$) such that $g(t)$ varies continuously. Then $h_{g(t)}(\xi)$ is jointly continuous in t and $\xi$.
\end{thm}

\begin{thm}\label{thm: Uniform Lipschitz cty of h^2}
Let $(M,g_0)$ be a compact Riemannian surface and $g(t)$ evolve under normalized Ricci flow using $g_0$ as the initial metric. Then $h^2_{g(t)}(\xi)$ (where $\xi\in[0,1]$) is uniform Lipschitz continuous with Lipschitz constant (in $\xi$) uniformly bounded by $4\pi|M|_{g_0}+4|K_0||M|^2_{g_0}$ over the time-interval $[0,+\infty]$, where $K$ is the Gauss curvature and $K_0=min\{\underset{(M,g_0)}{\text{inf}\;}K, \text{0}\}$.
\end{thm}

\begin{cor}\label{cor: Uniform Lipschitz cty of h}
Let $(M,g_0)$ be a compact Riemannian surface and $g(t)$ evolve under normalized Ricci flow using $g_0$ as the initial metric. Then for any $T \in [0, \infty]$, for any fixed compact subinterval $[\xi_{0},\xi_{1}]\subseteq(0,1)$, the Lipschitz constant of $h_{g(t)}(\xi)$ (in $\xi$) is uniformly bounded over the time-interval $[0,T]$ by $2\alpha^{-1}( \pi|M|_{g_0}+|K_0||M|^2_{g_0})$, where $\alpha:=inf\{h_{g(t)}(\xi):t\in [0,T],\xi \in [\xi_0,\xi_1]\}>0$ and $K_0=min\{\underset{(M,g_0)}{\text{inf}\;}K, \text{0}\}$.
\end{cor}

\begin{rmk}\label{rmk: geometric meaning of boundedness of Lipschitz constant of h}
Corollary \ref{cor: Uniform Lipschitz cty of h} has a geometric meaning. Namely, at each time $t$ under normalized Ricci flow on a compact Riemannian surface, if one picks an isoperimetric region $\Omega$ whose area ratio lies in $[\xi_{0},\xi_{1}]$, then the absolute value of geodesic curvature of $\partial \Omega$ is uniformly bounded by $2\alpha^{-1}( \pi|M|_{g_0}+|K_0||M|^2_{g_0})$.
\end{rmk}

The paper is organized as follows. In section 2 we first review some definitions and then prove the comparability of perimeters of a set of finite perimeter under equivalent metrics. Moreover we prove one compactness theorem for sets of finite perimeter with uniform perimeter bound under continuous convergence of metrics on compact manifolds. In section 3, we prove main theorem \ref{thm: joint cty of h}. In section 4, we prove the uniform asymptotic behavior of $h_{g(t)}(\xi)$ as $\xi \rightarrow{0}$ over time-interval $[0,T]$ on compact surfaces. Lastly, main theorem \ref{thm: Uniform Lipschitz cty of h^2} and corollary \ref{cor: Uniform Lipschitz cty of h} are proved in section 5.

\subsection{Acknowledgements}
This work is part of my doctoral thesis work under the supervision of Professor Hans-Joachim Hein and Professor Bianca Santoro. I gratefully acknowledge their patient guidance, constant encouragement and help. I would also like to thank Doctor Pei-Ken Hung for many helpful discussions when I started this project. My thanks also go to Professor Simon Brendle, Professor Gao Chen and Professor Richard Hamilton for their suggestions.

\section{Definitions and preliminary results}

In this section, we first review the definition of sets of finite perimeter, which was introduced by Caccioppoli in 1927 and De Giorgi in 1950s, and then we prove a comparability theorem and a compactness theorem mentioned above.

\begin{defi}
Given a Riemannian manifold $(M^n,g)$, we say a set $E \subseteq M$ is \textit{measurable} if it is measurable with respect to the canonical Riemannian measure induced by the metric g. Let $\chi_E$ denote the characteristic function of set $E$.
\end{defi}

\begin{rmk}
All Riemannian manifolds $M$ in our discussion are assumed to be connected, without boundary and of dimension $\geq{2}$ unless otherwise stated. All subsets of $M$ in our discussion are assumed to be measurable.
\end{rmk}

\begin{defi}
Given a function $u\in L^1(M,g)$, define the \textit{(total) variation} of $u$ by
$$
|Du|_g(M):=\text{sup}\;\{\intop_{M}{u\; \text{\text{div}}(Y)dV_g}:Y\in\mathfrak{X}_c(M), \;|Y|_g\leqslant{1}\},
$$ where $\mathfrak{X}_c(M)$ denotes the set of all smooth vector fields on $M$ with compact support.

A function $u\in L^1(M)$ is of \textit{bounded variation} if its variation is finite. Let $BV(M,g)$ denote the set of all functions of bounded variation on $M$. Clearly, $BV(M,g)\subseteq{L^1(M,g)}$.
\end{defi}

\begin{rmk}\label{rmk: variation of C^1 functions}
(i) If $u\in C^{1}_c(M)$, then $|Du|_g(M)=\intop_{M}{|\nabla u|dV_g}$.

(ii) The map $u\longmapsto |Du|_g(M)$ is $L^1$ lower semi-continuous.
\end{rmk}

\begin{defi}
Let $E\subseteq{M}$ have finite volume, i.e. $\chi_E\in L^1(M,g)$. We define the \textit{perimeter} of $E$ to be the variation of $\chi_E$
$$
P_g(E):=|D\chi_E|(M)=\text{sup}\;\{\intop_{E}\text{div}(Y)dV_g:Y\in\mathfrak{X}_c(M),\;|Y|_g\leqslant1\}.
$$
If $P_g(E)<\infty$, then we say $E$ is a \textit{set of finite perimeter}.
\end{defi}

\begin{rmk}
(i) If $\partial E$ is smooth, then $P_g(E)=|\partial E|_g$ by Stokes' Theorem.

\noindent (ii) If $N\subseteq{M}$ is a set of measure zero, then $P_g(N)=0$, so $P_g(E \cup N)=P_g(E)$.

\noindent (iii) It is easy to show that $P_g(E)=P_g(E^c)$, where $E^c=M\backslash{E}$.
\end{rmk}

Let ${\mathbb{P}}(M,g):=\{\text{all measurable subsets of (M,g)}\}$. We define the distance between two sets $E$ and $G$ in $\mathbb{P}$ by $d(E,G):=|E\Delta G|_g$, where $\Delta$ denotes the symmetric difference and we consider two sets $E$ and $G$ to be equivalent if and only if $|E\Delta G|_g=0$. Then $({\mathbb{P}}(M,g), d)$ is a metric space.

\begin{defi}
We say a sequence $E_j$ of sets of finite perimeter \textit{converges to a set $E$ in $L^1$}, written  $E_j\rightarrow{E}$ in $L^1(M,g)$, if $|E_j \Delta E|_g\rightarrow{0}$ or equivalently $\chi_{E_j}\rightarrow{\chi_E}$ in $L^1(M,g)$.
\end{defi}

\begin{rmk}
(i) The metric space $({\mathbb{P}}(M,g), d)$ is complete.

(ii) If $E_j\rightarrow{E}$ in $L^1(M,g)$, then $E$ has finite volume, i.e. $\chi_E \in L^1(M,g)$. So it makes sense to talk about $P(E)$ (though it may be infinite).

(iii) It is not hard to see that the perimeter function is lower semi-continuous with respect to the $L^1$ convergence.
\end{rmk}

\subsection{Smooth Approximation to Sets of Finite Perimeter on Manifolds}

\indent In the Euclidean setting, it is well known that domains with smooth boundary are dense in the sets of finite perimeter. For example, see Theorem 3.42 of the book \cite{AFP}. What we need is the Riemannian-manifold analogue. First, we need a useful lemma.

\begin{lem}[Proposition 1.4 of \cite{JPPP}]\label{lem: smooth approximation of functions to characteristic function of set of finite perimeter} For every $u\in BV(M)$, there exists a sequence $u_k\in C_c^{\infty}(M)$ such that $u_k \rightarrow{u}$ in $L^1(M,g)$ and $|Du|(M)=\underset{k\rightarrow{\infty}}{\text{lim}\;}\intop_{M}{|\nabla u_k|dV_g}$.
\end{lem}

\begin{rmk}\label{rmk: range of u_k can be in [0,1]}
If $u=\chi_E\in BV(M)$ in the above lemma for some $E\subseteq{M}$, then the convergent sequence $u_k$ can be chosen to additionally satisfy $0\leqslant u_k \leqslant 1$ for all $k$.
\end{rmk}

Once we have the above lemma, we can follow the proof of the Euclidean case to get the following density result on manifolds.

\begin{defi}\label{d:iso_profile}
Given a Riemannian manifold $(M,g)$, let's denote $$\mathcal{F}_g:=\{E \subseteq{M}: P_g(E)<\infty\} \;\;\text{and}\;\; \widetilde{\mathcal{F}}_g:=\{E \subseteq{M}: \partial E\; \text{is smooth and}\;|\partial E|<\infty\}.$$
\end{defi}

\begin{prop}\label{prop:desity of smooth boundary sets}
$\widetilde{\mathcal{F}}$ is dense in $\mathcal{F}$. More precisely, given any $E \in \mathcal{F}$, there exists a sequence $E_j \in \widetilde{\mathcal{F}}$ such that
$$
E_j \rightarrow{E} \;\;in\;\;L^1(M,g)\;\;and\;\;|\partial E_j|_g\rightarrow{P_g(E)}.
$$
\end{prop}

Applying Lemma \ref{lem: smooth approximation of functions to characteristic function of set of finite perimeter} and remark \ref{rmk: range of u_k can be in [0,1]}, we can prove the following proposition.

\begin{prop}\label{prop: perimeter ineq for sets intersection and union}
Given a Riemannian manifold $(M^n,g)$ and two sets of finite perimeter $A$ and $B$, we have $P_g(A \cup B) + P_g(A \cap B) \leqslant P_g(A) + P_g(B)$.
\end{prop}

\subsection{Sets of Finite Perimeter under Equivalent Metrics}

\indent Before we study the behavior of sets of finite perimeter under continuous variation of Riemannian metrics, it is important to first study the behavior of sets of finite perimeter on a Riemannian manifold when one changes the equipped metric to another equivalent one.

\begin{thm}[Comparability of perimeters under equivalent metrics]\label{thm: comparability of perimeter V1}
Let $(M^n,g_1)$ be a Riemannian manifold and $E \subseteq{M}$ be a set of finite perimeter with respect to $g_1$. If $g_2$ is a metric equivalent to $g_1$, i.e. $\frac{1}{C}g_1\leqslant{g_2}\leqslant{Cg_1}$ for some constant $C\geq{1}$, then 
$$C^{-\frac{n-1}{2}}P_{g_1}(E)\leqslant P_{g_2}(E)\leqslant
C^{\frac{n-1}{2}}P_{g_1}(E).
$$
\end{thm}
\begin{proof}
Given $P_{g_1}(E)<\infty$, then by Proposition \ref{prop:desity of smooth boundary sets}, we know there exists a sequence of smooth domains $E_k$ such that $E_k \rightarrow{E}$ in $L^1(M,g_1)$ and $|\partial E_k|_{g_1}\rightarrow{P_{g_1}(E)}$. Then $E_k \rightarrow{E}$ in $L^1(M,g_2)$ because
$$
|E_k\;\Delta\;E|_{g_2}\leqslant{C^{\frac{n}{2}}|E_k\;\Delta\;E|_{g_1}}\rightarrow{0}.
$$
Note that since each $\partial E_k$ is smooth, the metrics induced by $g_1$ and $g_2$ on each $\partial E_k$ are also equivalent. In particular, we have ${C^{-\frac{n-1}{2}}|\partial E_k|_{g_1}}\leqslant |\partial E_k|_{g_2}\leqslant{C^{\frac{n-1}{2}}|\partial E_k|_{g_1}}$ for each $k$. Then lower semi-continuity of perimeter implies that 
$$P_{g_2}(E)\leqslant{\underset{k\rightarrow{\infty}}{\text{liminf\;}}}P_{g_2}(E_k)=\underset{k\rightarrow{\infty}}{\text{liminf\;}}|\partial E_k|_{g_2}\leqslant \underset{k\rightarrow{\infty}}{\text{liminf\;}}C^{\frac{n-1}{2}}|\partial E_k|_{g_1}=C^{\frac{n-1}{2}}P_{g_1}(E).
$$
Now we know $P_{g_2}(E)<\infty$, so that $\chi_{E}\in BV(M,g_2)$. Then Lemma \ref{lem: smooth approximation of functions to characteristic function of set of finite perimeter} tells us that there exists a sequence of smooth functions $u_k \in C^{\infty}_{c}(M)$ such that 
$$
u_k \rightarrow{\chi_{E}}\;\; in \;\; L^1(M,g_2) \;\; and\;\; P_{g_2}(E)=|D\chi_{E}|_{g_2}(M)=\underset{k\rightarrow{\infty}}{\text{lim}\;}\intop_{M}{|\nabla u_k|_{g_2}dV_{g_2}}.
$$
If the following claim holds, then we are done.

\noindent \textbf{Claim:} $\underset{k\rightarrow{\infty}}{\text{lim}\;}\intop_{M}{|\nabla u_k|_{g_2}dV_{g_2}}\geq{C^{-\frac{n-1}{2}}P_{g_1}(E)}$. 

\noindent Proof of the Claim. First, note that $u_k \rightarrow{\chi_{E}}\;\text{in} \;L^1(M,g_2)$ implies that $u_k \rightarrow{\chi_{E}}\;\text{in} \;L^1(M,g_1)$ because 
$$
\intop_{M}{|u_k-\chi_E|dV_{g_1}}\leqslant \intop_{M}{C^{n/2}|u_k-\chi_E|dV_{g_2}}\rightarrow{0}.
$$
Since $u_k \in C^{\infty}_{c}(M)$, remark \ref{rmk: variation of C^1 functions} gives $|Du_k|_{g_1}(M)=\intop_{M}{|\nabla u_k|_{g_1}dV_{g_1}}$. Then lower semi-continuity of variation implies
$$
\underset{k\rightarrow{\infty}}{\text{liminf\;}}\intop_{M}{|\nabla u_k|_{g_1}dV_{g_1}}=\underset{k\rightarrow{\infty}}{\text{liminf\;}}|Du_k|_{g_1}(M)\geqslant{|D\chi_E|_{g_1}(M)}=P_{g_1}(E)
$$ 
Thus
$$
\underset{k\rightarrow{\infty}}{\text{lim}\;}\intop_{M}{|\nabla u_k|_{g_2}dV_{g_2}}\geq{\underset{k\rightarrow{\infty}}{\text{liminf\;}}\intop_{M}{C^{1/2}|\nabla u_k|_{g_1}C^{-n/2}dV_{g_1}}}\geqslant{C^{-\frac{n-1}{2}}P_{g_1}(E)},
$$as claimed.
\end{proof}

\begin{cor}
Let $(M,g_1)$ be a Riemannian manifold and $E \subseteq{M}$ be a measurable subset. If $g_2$ is equivalent to $g_1$, then $P_{g_1}(E)<\infty$ if and only if $P_{g_2}(E)<\infty$.
\end{cor}

\begin{cor}
If $g_1$ and $g_2$ are two equivalent metrics on a manifold $M$, then
$\mathcal{F}_{g_1} = \mathcal{F}_{g_2}$ and $\widetilde{\mathcal{F}}_{g_1}=\widetilde{\mathcal{F}}_{g_2}$.
\end{cor}

\begin{cor}
Let $g_1$ and $g_2$ be any two Riemannian metrics on a compact manifold $M$. Then
$\mathcal{F}_{g_1} = \mathcal{F}_{g_2}$ and $\widetilde{\mathcal{F}}_{g_1}=\widetilde{\mathcal{F}}_{g_2}$.
\end{cor}
\begin{proof}
This follows from the fact that every two metrics on a compact manifold are equivalent.
\end{proof}

\subsection{The Isoperimetric Profile and Compactness for Sets of Finite Perimeter on Compact Manifolds}

\begin{defi}\label{d:iso_profile}
Given a compact Riemannian manifold $(M,g)$ (which has finite volume), let's consider the following two functions, for volume ratio $\xi \in (0,1)$,
$$\widetilde{h}(\xi):=\text{inf}\;\{|\partial\Omega|_g:\;|\Omega|_g=\xi|M|_g,\;\partial\Omega\;\text{is smooth}\},$$
and
$$h(\xi):=\text{inf}\;\{P_g(\Omega):\;|\Omega|_g=\xi|M|_g,\;P_g(\Omega)<\infty\}.$$

A set $\Omega$ of finite perimeter that attains the infimum for $h(\xi)$ is called an \textit{isoperimetric region} for $\xi$ (or for $(g,\xi)$). In this case  $\partial \Omega$ is called an \textit{isoperimetric hypersurface}.
\end{defi}

The Euclidean version of compactness theorem for sets of finite perimeter is well known (see Theorem 12.26 of \cite{Mag}). In our application, what we need is the Riemannian-manifold analogue under continuous convergence of metrics. We start with the following special case of a fixed metric which can be proved by (i) fixing a finite family of coordinate charts $\{U_i,\varphi_i\}_{i=1}^{N}$ that covers $M$ such that each $U_i$ has a Lipschitz boundary and then (ii) invoke both Proposition \ref{prop: perimeter ineq for sets intersection and union} and the Euclidean version of compactness theorem in each chart by noting that the induced metric on all charts are uniformly equivalent to the Euclidean metric, and finally (iii) carefully patch up the resulting sets together.

\begin{lem}\label{lem: Compactness for sets of finite perimeter}

Let $(M,g)$ be a compact Riemannian manifold. If $\Omega_k\subseteq{M}$ is a sequence of sets of finite perimeter with uniform bounded perimeter, i.e. $P(\Omega_k)\leqslant{C}$, then there is a subsequence $\Omega_{k_j}$ and a set $W\subseteq{M}$ of finite perimeter such that
\begin{enumerate}[(1)]
    \item $\Omega_{k_j}\rightarrow{W}$ in $L^1(M,g)$ or equivalently $|\Omega_{k_j}\Delta W|_g\rightarrow0$ as $j\rightarrow\infty$, and
    \item $|\Omega_{k_j}|_g\longrightarrow |W|_g$ as $j\rightarrow\infty$, and
    \item $\underset{j\rightarrow\infty}{liminf}\;P_g(\Omega_{k_j})\geqslant P_g(W)$.
\end{enumerate}
\end{lem}

\begin{rmk}
This lemma says $\{E \subseteq{M}: P_g(E)\leqslant{C}\}$ is a compact subset of $({\mathbb{P}}(M,g), d)$ when $M$ is a compact manifold.
\end{rmk}

Applying the above compactness result to minimizing sequences for $h(\xi)$, we can easily prove the following well-known existence of isoperimetric regions on compact manifolds.

\begin{prop}[Existence of isoperimetric regions]\label{prop:Existence of isoperimetric region}
Let $(M,g)$ be a compact Riemannian manifold. For any $\xi \in (0,1)$, there exists an isoperimetric region for $\xi$. That is to say, there exists a set $\Omega$ of finite perimeter such that $|\Omega|_g=\xi|M|_g$ and $P_g(\Omega)=h(\xi)$.
\end{prop}

Using the existence of isoperimetric regions and the density of smooth domains in sets of finite perimeter, it is not hard to show that $h(\xi)=\widetilde{h}(\xi)$, for all $\xi \in (0,1)$.

\begin{defi}
Since $h(\xi)=\widetilde{h}(\xi)$, we can just call them 
the \textit{isoperimetric profile} of $(M,g)$, denoted by $h_{g}(\xi)$ or $h(\xi)$.
\end{defi}

Now, let's prove the compactness theorem with continuous convergent metrics, which is the key ingredient of proving joint continuity of the isoperimetric profile in next section.

\begin{thm}[Compactness for sets of finite perimeter with convergent metrics]\label{lem:Compactness for sets of finite perimeter with changing metrics}

Let $(M^n,g(t_0))$ be a compact Riemannian manifold and $g(t)$ be a family of metrics converging to ${g(t_0)}$ continuously. If $t_k \rightarrow t_0$ and $\Omega_k\subseteq{M}$ is a sequence of sets of finite perimeter with uniform bounded perimeter, i.e. $P_{g(t_k)}(\Omega_k)\leqslant{C}$ (the constant C is independent of $\Omega_k$ and $g(t_k)$), then there is a subsequence $\Omega_{k_j}$ and a set $W\subseteq{M}$ of finite perimeter (with respect to $g(t_0)$) such that
\begin{enumerate}[(1)]
    \item $|\Omega_{k_j}\Delta W|_{g(t_{k_j})}\rightarrow0$ as $j\rightarrow\infty$, and
    \item $|\Omega_{k_j}|_{g(t_{k_j})}\longrightarrow |W|_{g(t_0)}$ as $j\rightarrow\infty$, and
    \item $\underset{j\rightarrow\infty}{\text{liminf\;}}(P_{g(t_{k_j})}(\Omega_{k_j}))\geqslant P_{g(t_0)}(W)$.
\end{enumerate}
\end{thm}

\begin{proof}
Since $g(t_k)\rightarrow{g(t_0)}$ continuously, for any $\epsilon>0$, there exists $N$ such that $k>N$ implies
\begin{equation}\label{eq:metric comparison}
(1+\epsilon)^{-2}g_{t_{0}}\leq g_{t_{k}}\leq (1+\epsilon)^{2}g_{t_{0}}\;\;\;\;\;\text{or}\;\;\;\;\;(1+\epsilon)^{-2}g_{t_{k}}\leq g_{t_{0}}\leq (1+\epsilon)^{2}g_{t_{k}}.
\end{equation}
Then Theorem \ref{thm: comparability of perimeter V1} implies, for $k>N$,
$$
(1+\epsilon)^{-(n-1)}P_{t_{k}}(\Omega_{k})\leq P_{t_{0}}(\Omega_{k})\leq (1+\epsilon)^{n-1}P_{t_{k}}(\Omega_{k}),
$$
which implies that $P_{t_{0}}(\Omega_{k})\leqslant{C}$. Hence we can apply Lemma \ref{lem: Compactness for sets of finite perimeter} to the metric $g(t_0)$ to get a subsequence $\Omega_{k_j}$ and a set $W$ such that
\begin{equation}\label{eq:temp1}
 \;\; |\Omega_{k_j}\Delta W|_{t_0}\rightarrow0 \;\;\text{as}\;\; j\rightarrow\infty,
\end{equation}
\begin{equation}\label{eq:temp2} |\Omega_{k_j}|_{t_0}\longrightarrow |W|_{t_0} \;\;\text{as}\;\; j\rightarrow\infty,
\end{equation}  
and
\begin{equation}\label{eq:temp3}
\underset{j\rightarrow\infty}{\text{liminf\;}}(P_{t_0}(\Omega_{k_j}))\geqslant P_{t_0}(W).
\end{equation}
Note that (\ref{eq:metric comparison}) and Theorem \ref{thm: comparability of perimeter V1} imply that
\begin{equation}\label{eq:temp4}
(1+\epsilon)^{-(n-1)}P_{t_0}(\Omega_{k_j})\leq P_{t_{k_j}}(\Omega_{k_j})\leq (1+\epsilon)^{n-1}P_{t_0}(\Omega_{k_j}),
\end{equation}
\begin{equation}\label{eq:temp5}
(1+\epsilon)^{-n}|\Omega_{k_j}|_{t_{0}}\leq |\Omega_{k_j}|_{t_{k_j}}\leq (1+\epsilon)^{n}|\Omega_{k_j}|_{t_{0}},
\end{equation}
and
\begin{equation}\label{eq:temp6}
(1+\epsilon)^{-n}|\Omega_{k_j}\Delta W|_{t_{0}}\leq |\Omega_{k_j}\Delta W|_{t_{k_j}}\leq (1+\epsilon)^{n}|\Omega_{k_j}\Delta W|_{t_{0}}.
\end{equation}
Now (\ref{eq:temp1}) and (\ref{eq:temp6}) imply
$$
|\Omega_{k_j}\Delta W|_{t_{k_j}}\rightarrow{0},
$$which is the desired property (1).

Letting $j\rightarrow{\infty}$, (\ref{eq:temp2}) and (\ref{eq:temp5}) imply
$$
(1+\epsilon)^{-n}|W|_{t_{0}}\leq \underset{j\rightarrow{\infty}}{\text{liminf\;}}|\Omega_{k_{j}}|_{t_{k_{j}}}\leq \underset{j\rightarrow{\infty}}{\text{limsup\;}}|\Omega_{k_{j}}|_{t_{k_{j}}}\leq(1+\epsilon)^{n}|W|_{t_{0}}.
$$
Now, letting $\epsilon\rightarrow0$ gives
$$
|\Omega_{k_j}|_{t_{k_j}}\longrightarrow |W|_{t_0} \;\;\text{as}\;\; j\rightarrow\infty,
$$which is the desired property (2).

Letting $j\rightarrow{\infty}$, (\ref{eq:temp3}) and (\ref{eq:temp4}) imply
$$
(1+\epsilon)^{-(n-1)}P_{t_{0}}(W)\leq \underset{j\rightarrow{\infty}}{\text{liminf\;}} P_{t_{k_{j}}}(\Omega_{k_{j}}).
$$
Now, letting $\epsilon\rightarrow0$ gives
$$
P_{t_{0}}(W)\leq \underset{j\rightarrow{\infty}}{\text{liminf\;}} P_{t_{k_{j}}}(\Omega_{k_{j}}),
$$which is the desired property (3).
\end{proof}

\section{Proof of the joint continuity of the isoperimetric profile}

\begin{thm}[Section 8.5 and Theorem 10.2 of \cite{Mor}. Theorem V.4.1 of \cite{Cha}. Proposition 2.4 of \cite{RR}]\label{thm: regularity of isoperimetric region}
Given a compact Riemannian manifold $(M^n,g)$, for each $\xi \in (0,1)$, there exists a corresponding isoperimetric region whose boundary, apart from a singular closed set of Hausdorff dimension at most $n-8$, is a smooth embedded orientable hypersurface with constant mean curvature. 
\end{thm}

\begin{proof}[\textbf{Proof of Theorem \ref{thm: joint cty of h}.}]
Fix any $t_0$ and $\xi_0$. By rescaling, we may assume $|M|_{g(t_0)}={1}$. Fix an isoperimetric region $\Omega_0$ for $(g(t_0),\xi_0)$. Let $(t_i,\xi_i)$ be a sequence converging to ${(t_0,\xi_0)}$. Let each $\Omega_{i}$ be an isoperimetric region for $(g(t_{i}),\xi_{i})$. Then
$$
|\Omega_{i}|_{t_{i}}=\xi_{i}|M|_{t_i}\;\;\text{and}\;\;|\partial\Omega_{i}|_{t_{i}}=h_{g(t_{i})}(\xi_{i}),
$$ where $|\cdot|_{t_i}$ is the boundary area or volume measurement with respect to the metric $g(t_i)$.

Denote $L_i=h_{g(t_{i})}(\xi_{i})$ and $L_0=h_{g(t_{0})}(\xi_{0})$. We need to show $L_i\rightarrow{L_0}$.

Denote $\mathring{L_{i}}=|\partial\Omega_{i}|_{t_{0}}$, $\mathring{\xi_{i}}=|\Omega_{i}|_{t_{0}}$, $L_{0,i}=|\partial\Omega_{0}|_{t_{i}}$ and $\xi_{0,i}=|\Omega_{0}|_{t_{i}}/|M|_{t_i}$. Since $|M|_{t_i}\rightarrow{|M|_{t_0}=1}$, it follows that
$$
L_{0,i}\rightarrow{L_0}\;\;\text{and}\;\;\xi_{0,i}\rightarrow{\xi_0}.
$$

\noindent {\bf Claim 1:} ${\text{limsup\;}}(L_i)\leqslant L_0$. 

\noindent {Proof of Claim 1.}

Note that, by Theorem \ref{thm: regularity of isoperimetric region}, $\partial \Omega_0$ is not dense in $M$. Then fixing any small enough $\epsilon>0$, we can do the following:
\begin{enumerate}
    \item pick a point $p\in{M}$ and a geodesic ball $B_{0}(p)$ centered at p for $(M, g(t_0))$ such that $B_{0}(p)$ is disjoint from $\Omega_0$ and $|B_{0}(p)|_{t_0}=2\epsilon$. Fix such p.
    \item pick a point $q\in{\Omega_0}$ and a geodesic ball $B_{0}(q)$ centered at q for $(M, g(t_0))$ such that $B_{0}(q)$ is fully contained in the interior of $\Omega_0$ and $|B_{0}(q)|_{t_0}=2\epsilon$. Fix such q.
\end{enumerate}

Because $(\xi_{0,i}-\xi_i)|M|_{t_i}=((\xi_{0,i}-\xi_0)+(\xi_0-\xi_i))|M|_{t_i}\rightarrow{0}$, for the above $\epsilon$, there exists a big enough number $N_{1}>0$ such that $i>N_{1}$ implies  $|\xi_{0,i}-\xi_i||M|_{t_i}<\epsilon$. On the other hand, convergence of metrics implies that there exist a big enough number $N_2>0$ such that we can construct a sequence of comparison regions $\Omega_0(B_i)$ in the following way: denote $\text{max}\;\{N_1,N_2\}$ by $N$,
\begin{enumerate}
    \item if $i<N$, then let $\Omega_0(B_i)=\Omega_i$.
    \item if $i\geq{N}$ and $\xi_i>\xi_{0,i}$, then there exists a geodesic ball $B_i(p)\subseteq{B_{0}(p)}$ centered at p with area $|B_i(p)|_{t_i}=(\xi_i-\xi_{0,i})|M|_{t_i}<\epsilon$ (using convergence of metrics). Let  $\Omega_0(B_i)=\Omega_0\cup{B_i(p)}$.
    \item if $i\geq{N}$ and $\xi_i<\xi_{0,i}$, then there exists be a geodesic ball $B_i(q)\subseteq{B_{0}(q)}$ centered at q with area $|B_i(q)|_{t_i}=(\xi_{0,i}-\xi_i)|M|_{t_i}<\epsilon$. Let  $\Omega_0(B_i)=\Omega_0\backslash{B_i(q)}$.
    \item if $i\geq{N}$ and $\xi_i=\xi_{0,i}$, then let $\Omega_0(B_i)=\Omega_0$.
\end{enumerate}

By construction, comparison region $\Omega_0(B_i)$ satisfies
$$
|B_i|_{t_i}=|\xi_{0,i}-\xi_i||M|_{t_i}\;\;\text{and}\;\;|\Omega_0(B_i)|_{t_i}=\xi_i|M|_{t_i}\;\;\text{for}\;\text{all}\;i>N,
$$
where $B_i$ are the above geodesic balls centered either at p or at q.

Since $\Omega_{i}$ is an isoperimetric region for $(g(t_{i}),\xi_{i})$ and $|\Omega_{i}|_{t_i}=|\Omega_0(B_i)|_{t_i}$, we have $|\partial\Omega_i|_{t_i}\leqslant{|\partial\Omega_0(B_i)|_{t_i}}$, that is
$$
L_i\leqslant{L_{0,i}+|\partial B_i|_{t_i}}\;\;\text{for}\;\text{all}\;i>N.
$$ Then taking limsup gives the desired inequality $\text{limsup\;}(L_i)\leqslant{L_0}$ because $|\partial B_i|_{t_i}\rightarrow{0}$.\newline

\noindent {\bf Claim 2:} ${\text{liminf\;}}(L_i)\geqslant L_0$.

\noindent {Proof of Claim 2.}
Suppose for contradiction that ${\text{liminf\;}}(L_i)<L_0$. Then there exists a subsequence $\Omega_{i_j}$ such that $\underset{j\rightarrow\infty}{\text{lim}\;}L_{i_{j}}<L_{0}$, which implies $|\partial \Omega_{i_j}|=L_{i_j}\leq C$. Hence Theorem \ref{lem:Compactness for sets of finite perimeter with changing metrics} implies that there is a further subsequence, still denoted by $\Omega_{i_j}$, and a set W of finite perimeter such that 
$$
|W|_{t_0}=\xi_0\;\;\text{and}\;\; \underset{j\rightarrow\infty}{\text{liminf\;}}|\partial\Omega_{i_j}|_{t_{i_j}}\geq{P(W)_{t_0}}.
$$
Therefore $\underset{j\rightarrow\infty}{\text{lim}\;}(L_{i_j})=\underset{j\rightarrow\infty}{\text{liminf\;}}|\partial\Omega_{i_j}|_{t_{i_j}}\geq{P(W)_{t_0}}\geq|\partial\Omega_0|_{t_0}={L_0}$. This gives a contradiction.
\end{proof}

\begin{rmk}
In the proof of claim 1, to show $\partial \Omega_0$ is not dense in $M$, we may alternatively use the so-called ``density estimates" for isoperimetric region $\Omega_0$, which is much lighter than the machinery of Theorem \ref{thm: regularity of isoperimetric region}. See more details in section 16.2 of \cite{Mag}.
\end{rmk}

\section{Uniform asymptotic behavior of the isoperimetric Profile}

The normalized Ricci flow on a compact Riemannian surface $(M,g_0)$ is the following nonlinear evolution equation for metrics 
\begin{equation}\label{eq: normalized RF on surfaces}
\begin{cases}
\frac{\partial g}{\partial t}=-2(K-\overline{K})g
\\ g(0)=g(t_0)
\end{cases}\;\;\;\;t\geqslant{0},
\end{equation}where $K$ is the Gauss curvature of $M$ and $\overline{K}=\frac{1}{|M|}\intop_{M}{Kd\mu(g)}$ is the average curvature.

We will make use of the following standard results of long time existence and uniqueness for normalized Ricci flow on compact surfaces.

\begin{thm}[\cite{Ham}]\label{thm: long time existence for NRF on surfaces}
Given a compact Riemannian surface $(M^2,g_0)$. If $\overline{K}_{g_0}\leqslant{0}$ or $K_{g_0}(p)\geqslant{0}$ for all $p\in M$, then the solution to (\ref{eq: normalized RF on surfaces}) exists for all time $t\geqslant{0}$ and converges smoothly to a metric of constant curvature.
\end{thm}

\begin{thm}[\cite{Cho}]\label{thm: thm: long time existence for NRF on S^2}
If $g_0$ is any metric on $S^2$, then its evolution under normalized Ricci flow develops positive Gauss curvature in finite time, and thus by Theorem \ref{thm: long time existence for NRF on surfaces} converges smoothly to the round metric as $t\rightarrow{\infty}$.
\end{thm}

Given any compact Riemannian surface $(M,g)$, we will see in next section, $h^{2}(\xi)+(\underset{(M,g)}{\text{inf}\;}K )(\xi|M|)^2$ is concave in the variable $\xi \in (0,1)$ (by Lemma \ref{lem: Lipschitz cty of h on surfaces}). Then we know the derivative of this concave function is ``biggest" when $\xi\rightarrow{0^+}$ and ``smallest" when $\xi\rightarrow{1^-}$. So if we know the asymptotic behavior of $h(\xi)$ near two endpoints, then we may get some useful uniform bound for $h'(\xi)$. Before exploring that, it is useful to observe that $h(\xi)=h(1-\xi)$; namely, $h(\xi)$ is symmetric about $\xi=\frac{1}{2}$, so we just need to know its asymptotic behavior as $\xi \rightarrow{0^+}$. Moreover, because we will evolve metrics $g(t)$ in our applications, what we actually need is the uniform asymptotic behavior of $h_{g(t)}(\xi)$ as $\xi \rightarrow{0^+}$ for all $t\in [0,T]$. Such asymptotic behavior for a fixed metric has been examined by Andrews and Bryan in \cite{AB}. What we prove in the following is the metric-varying version of their theorem.

\begin{lem}[Theorem 4.3 of \cite{OSS}. Isoperimetric inequality on surfaces of variable curvature]\label{lem:Isoperimetric inequality on surfaces} Let $M$ be a Riemannian surface of variable Gauss curvature $K$ and let $D\subseteq M$ be a simply connected domain with finite area $A$ and denote $|\partial D|$ by $L$. Then
$$
L^2\geqslant{4\pi A-(\underset{D}{\text{sup}\;}K)A^{2}}.
$$
\end{lem}

\begin{thm}\label{thm: error coeff control of h(g(t))}
Let $g(t)$ be a family of Riemannian metrics on a compact surface $M$ that vary smoothly over $[0,T]$ for some $0\leqslant T \leqslant +\infty$. Then the isoperimetric profiles satisfy
$$
h_{g(t)}(\xi)=\sqrt{4\pi|M|_{g(t)}\xi}-\frac{|M|_{g(t)}^{3/2}\text{sup}_{M}K}{4\sqrt{\pi}}\xi^{3/2}+O(\xi^{2})\qquad as\;\xi\rightarrow0.
$$
Moreover, 

(1) The coefficient $C(t)$ in the error term $O(\xi^2)$ is bounded over $[0,T]$ by a constant $C(T)$. 

(2) If $g(t)$ evolves under the normalized Ricci flow using $g_0=g(0)$ as the initial metric, then $C(t)$ is bounded over $[0,+\infty]$ by a uniform constant $C_0$ depending only on $|M|_{g_0}$, $\underset{(M,g_0)}{\text{sup}\;}K$ and $\underset{(M,g_0)}{\text{sup}\;}\Delta K$.
\end{thm}

\begin{proof}
Let's prove the upper bound first. For any fixed metric $g(t)$, small geodesic balls centered at any point $p\in M$ are admissible regions in the definition of $h$. The result of [\cite{Gray}, Theorem 3.1] gives
\begin{equation}\label{eq:expansion of vol of  geodesic ball}
|B_{r}(p)|=\pi r^{2}(1-\frac{K(p)}{12}r^{2}+O(r^{4}))\;\;\text{and}\;\;|\partial B_{r}(p)|=2\pi r(1-\frac{K(p)}{6}r^{2}+O(r^{4})).
\end{equation}

By definition of $h$ we have $h(|B_{r}(p)|/|M|) \leqslant |\partial B_{r}(p)|$, which leads to the desired upper bound as follows. Let 
$$
\xi(p)=\frac{|B_{r}(p)|}{|M|}=\frac{\pi r^{2}}{|M|}(1-\frac{K(p)}{12}r^{2}+O(r^{4})),
$$
then
$$
r(\xi)=\sqrt{\frac{|M|}{\pi}\xi}(1+\frac{K|M|}{24\pi}\xi+O(\xi^{2})).
$$
Putting the expression for $r(\xi)$ into $|\partial B_{r}(p)|$ in (\ref{eq:expansion of vol of  geodesic ball}) we have
\begin{equation*}
\begin{split}
h(\xi)	& \leqslant|\partial B_{r}(p)| = \sqrt{4\pi|M|\xi}-\frac{|M|^{3/2}K}{4\sqrt{\pi}}\xi^{3/2}+O(\xi^{5/2}).
	\end{split}
\end{equation*}

Since the inequality holds for all points $p\in M$, we have
\begin{equation*}
\begin{split}
    h(\xi) \leqslant\sqrt{4\pi|M|\xi}-\frac{|M|^{3/2}\text{sup}_{M}K}{4\sqrt{\pi}}\xi^{3/2}+O(\xi^{2})\qquad as\;\xi\rightarrow0,
	\end{split}
\end{equation*}
which is the desired upper bound for any metric.

In order to compute the coefficient in the error term, we need one more term in the expansion of both $|\partial B_{r}(p)|$ and $|B_{r}(p)|$.

Again the result of [\cite{Gray}, Theorem 3.1] gives
$$
|\partial B_{r}(p)|=2\pi r(1-\frac{K(p)}{6}r^{2}+\frac{1}{240}(2K^{2}-3\Delta K)r^{4}+O(r^{6}))
$$
$$
|B_{r}(p)|=\pi r^{2}(1-\frac{K(p)}{12}r^{2}+\frac{1}{720}(2K^{2}-3\Delta K)r^{4}+O(r^{6})).
$$
If we use these expansions and follow the above computations, we can get
$$
|\partial B_{r}(p)|=\sqrt{4\pi|M|\xi}-\frac{K}{4\sqrt{\pi}}|M|^{3/2}\xi^{3/2}-\frac{6\Delta K+17K^{2}}{288\pi^{3/2}}|M|^{5/2}\xi^{5/2}+O(\xi^{7/2}).
$$
Hence we can choose (thanks to Taylor's theorem with remainder) $$C_1(t)=\underset{(M,g(t))}{\text{sup}\;}\frac{6\Delta K+17K^{2}}{288\pi^{3/2}}|M|_{g(t)}^{5/2}$$ as the coefficient in the error term $O(\xi^2)$ in the upper bound case.\newline

Secondly, let's prove the lower bound. Let's fix a metric $g$ for the moment. Since $M$ is compact, $\text{inj}(M,g)>0$, where $\text{inj}(M,g)$ is the injectivity radius of $(M,g)$. Then choose $\xi$ small enough such that $h(\xi)<\text{inj}(M,g)$. Let $\Omega$ be an isoperimetric region for $\xi$. We distinguish two cases.

\noindent {\bf Case 1}: $\Omega$ is connected.

Fix a point $p\in\partial\Omega$. Let $B$ denote the geodesic ball centered at $p$ with radius $h(\xi)$. Then $\Omega\subseteq B$ and $B$ is simply connected. Together with $\Omega$ being connected, we know $\Omega$ topologically is a disc with finitely many discs removed. Let $\widetilde{\Omega}$ denote the interior of the external
boundary of $\Omega$; namely, $\widetilde{\Omega}$ is $\Omega$ with the “holes” filled in. Then $\widetilde{\Omega}\subseteq{B}$ is simply connected, $|\Omega|\leqslant|\widetilde{\Omega}|$ and $|\partial\widetilde{\Omega}|\leqslant{|\partial\Omega|}$.

\indent {\bf Subcase 1}: $\text{sup}_M K \leqslant{0}$.

Lemma \ref{lem:Isoperimetric inequality on surfaces} gives
$$
|\partial \Omega|^2
\geqslant |\partial \widetilde{\Omega}|^2
\geqslant 4\pi |\widetilde{\Omega}|-(\underset{\widetilde{\Omega}}{\text{sup}\;}K)|\widetilde{\Omega}|^{2}
\geqslant 4\pi |\widetilde{\Omega}|-(\underset{M}{\text{sup}\;}K)|\widetilde{\Omega}|^{2}
\geqslant 4\pi |{\Omega}|-(\underset{M}{\text{sup}\;}K)|{\Omega}|^{2}.
$$

\indent {\bf Subcase 2}: $\text{sup}_M K >{0}$.

Let's further require $\xi\ll1/(2\pi |M| \text{sup}_M K)$ to ensure that $h^2(\xi)\ll\frac{2}{\text{sup}_M K}$. And thus $|B|<\frac{2\pi}{\text{sup}_M K}$. Then
$$
|\Omega|\leqslant|\widetilde{\Omega}|\leqslant{|B|}<{\frac{2\pi}{\text{sup}_M K}}\;\;\;\;\text{and}\;\;\;\;|\partial\widetilde{\Omega}|\leqslant{|\partial\Omega|}.
$$

Lemma \ref{lem:Isoperimetric inequality on surfaces} again gives
$$
|\partial \Omega|^2
\geqslant |\partial \widetilde{\Omega}|^2
\geqslant 4\pi |\widetilde{\Omega}|-(\underset{\widetilde{\Omega}}{\text{sup}\;}K)|\widetilde{\Omega}|^{2}
\geqslant 4\pi |\widetilde{\Omega}|-(\underset{M}{\text{sup}\;}K)|\widetilde{\Omega}|^{2}
\geqslant 4\pi |{\Omega}|-(\underset{M}{\text{sup}\;}K)|{\Omega}|^{2}.
$$

So in both subcases, we get
\begin{equation*}
    \begin{split}
        h(\xi) & \geqslant \sqrt{4\pi |M|\xi-(\underset{M}{\text{sup}\;}K)(|M|\xi)^{2}} \\ & =\sqrt{4\pi |M|\xi}-\frac{\text{sup}_{M}K}{4\sqrt{\pi}}(|M|\xi)^{3/2}-\frac{(\text{sup}_{M}K)^2}{64\pi^{3/2}}(|M|\xi)^{5/2}+O(\xi^{7/2})
    \end{split}
\end{equation*}

So case 1 is done.\newline

\noindent {\bf Case 2}: $\Omega$ is disconnected.

By Theorem \ref{thm: regularity of isoperimetric region}, $\partial \Omega$ has at most finitely many components. So without loss of generality, suppose $\Omega$ has two disjoint connected components, say $W_{1}$ and $W_{2}$. Then
$$
h(\xi)	=|\partial\Omega|=|\partial W_{1}|+|\partial W_{2}|\;\;\text{and}\;\;
\xi	=\xi_{1}+\xi_{2},\;\text{where}\;\xi_{i}=\frac{|W_{i}|}{|M|},i=1,2,
$$
and
$$
h(\xi)< \text{inj}(M,g)\Longrightarrow|\partial W_{1}|,|\partial W_{2}|< \text{inj}(M,g).
$$

Fix two points $p_{1}\in\partial W_{1}$ and $p_{2}\in\partial W_{2}$. Then the argument in case 1 can be applied to $W_{1}$ and $W_{2}$ respectively to get
$$
|\partial W_{1}|	\geqslant\sqrt{4\pi|M|\xi_{1}}-\frac{|M|^{3/2}\text{sup}_{M}K}{4\sqrt{\pi}}\xi_{1}^{3/2}-\frac{|M|^{5/2}(\text{sup}_{M}K)^2}{64\pi^{3/2}}{\xi_1}^{5/2}+O(\xi_1^{7/2})
$$
$$
|\partial W_{2}|	\geqslant\sqrt{4\pi|M|\xi_{2}}-\frac{|M|^{3/2}\text{sup}_{M}K}{4\sqrt{\pi}}\xi_{2}^{3/2}-\frac{|M|^{5/2}(\text{sup}_{M}K)^2}{64\pi^{3/2}}{\xi_2}^{5/2}+O(\xi_2^{7/2}).
$$

Then
\begin{equation*}
    \begin{split}
h(\xi)	& =|\partial W_{1}|+|\partial W_{2}|
	\\ & \geqslant\sqrt{4\pi|M|}(\sqrt{\xi_{1}+\xi_{2}})-\frac{|M|^{3/2}}{4\sqrt{\pi}}\text{sup}_{M}K(\xi_{1}^{3/2}+\xi_{2}^{3/2})-\frac{|M|^{5/2}(\text{sup}_{M}K)^2}{64\pi^{3/2}}({\xi_1}^{5/2}+{\xi_2}^{5/2})+O(\xi^{7/2})
    \end{split}
\end{equation*}

With the help of the following two elementary inequalities:
$$
x^{\gamma}+y^{\gamma}	\leqslant(x+y)^{\gamma}\;\;\;\;\;\text{for all}\; \gamma \geqslant{1},\; x,y\geqslant0
$$
$$
\alpha(\sqrt{x}+\sqrt{y})+\beta(x^{3/2}+y^{3/2})	\geqslant\alpha\sqrt{x+y}+\beta(x+y)^{3/2}\;\;\;\;\text{for all small enough}\;\;x,y\geqslant0
$$
where $\alpha,\beta>0$ are fixed constants, we can see that, no matter the sign of $\text{sup}_{M}K$, we have (make the $\xi$ smaller if needed)
$$
h(\xi)\geqslant\sqrt{4\pi|M|\xi}-\frac{|M|^{3/2}\text{sup}_{M}K}{4\sqrt{\pi}}\xi^{3/2}-\frac{|M|^{5/2}(\text{sup}_{M}K)^2}{64\pi^{3/2}}{\xi}^{5/2}+O(\xi^{7/2}).
$$ 

So case 2 is done and we have proved the desired lower bound for $h(\xi)$.\newline

It's clear that we can choose $$C_2(t)=\frac{(\text{sup}_{M}K)^2}{64\pi^{3/2}}|M|_{g(t)}^{5/2}$$ as the coefficient in the error term $O(\xi^2)$ in the lower bound case.

Therefore, now we can use $C(t)=\text{max}\;\{C_1(t),C_2(t)\}$ as the coefficient in the error term $O(\xi^2)$  for $h_{g(t)}(\xi)$.

Under smooth variation of metrics $g(t)$ over compact time interval $[0,T]$, from \cite{Ehr} we know that injectivity radius of $(M,g(t))$ is lower semi-continuous. So our argument in the lower bound case doesn't break down. Also $|M|$, $\underset{M}{\text{sup}\;}K$, and $\underset{M}{\text{sup}\;}(\Delta K+K^2)$ all change continuously and hence all are bounded, so $C(t)$ is also bounded over $[0,T]$.

If $g(t)$ evolves under normalized Ricci flow, we know the flow exists for all time $t\in [0,+\infty]$. And along the flow, $|M|$ is unchanged, $\text{sup}_M K$ is non-increasing and $\text{sup}_M \Delta K$ is also non-increasing. Thus $C_0=C(0)\geqslant{C(t)}$ for all $t\geqslant{0}$, as desired.
\end{proof}

\begin{rmk}
Now we see that $h(1)=h(0):=0$ is a continuous extension of the isoperimetric profile $h$ when $M$ is a compact surface.
\end{rmk}

\section{Proof of Theorem \ref{thm: Uniform Lipschitz cty of h^2} and Corollary \ref{cor: Uniform Lipschitz cty of h}}

We first need a lemma. 

\begin{lem}\label{lem: Lipschitz cty of h on surfaces}
Let $(M,g)$ be a compact Riemannian surface. Denote $\underset{(M,g)}{\text{inf}\;}K$ by $K_0$. Then the function $h^{2}(\xi)+K_0(\xi|M|)^2$ is concave in $\xi\in(0,1)$ and $h^{2}(\xi)+(K_0-\epsilon)(\xi|M|)^2$ is strictly concave for any $\epsilon>0$. Hence $h(\xi)$ is locally Lipschitz continuous. Moreover, if $h(\xi)$ is differentiable at $\xi$, then $h'(\xi)=\kappa|M|$, where $\kappa$ is the constant geodesic curvature of boundary curves of any isoperimetric region for $\xi$.
\end{lem}

\begin{proof}
This follows easily from Theorem V.4.3 of \cite{Cha} and Theorem \ref{thm: regularity of isoperimetric region}.
\end{proof}

For completeness, we state Theorem V.4.3 of \cite{Cha} in the following.

\begin{prop}[\cite{BP}. Theorem V.4.3 of \cite{Cha}]\label{prop: weak derivative of isoperimetric profile in the barrier sense}
Let $(M^n,g)$ be a compact Riemannian manifold and $\Omega \subseteq{M}$ be an isoperimetric region with smooth boundary and of constant mean curvature $\kappa$ with respect to $-\nu$, where $\nu$ is the unit exterior unit normal vector field along $\partial \Omega$. Then the isoperimetric profile $h(\xi)$ has weak left- and right- first derivatives and second derivative in the barrier sense satisfying
$$
\frac{D^+h}{d\xi}\leqslant \kappa|M| \leqslant \frac{D^-h}{d\xi}\;\;\;\;at\;\;\xi=\frac{|\Omega|}{|M|},
$$
and
$$
\frac{D^2h}{d\xi^2}\leqslant -\frac{|M|^2}{h^2(\xi)}\intop_{\partial \Omega}{||\mathrm{II}||^2 + Ric(\nu,\nu)dA}\;\;\;\;\;at\;\;\xi=\frac{|\Omega|}{|M|},
$$where $\mathrm{II}$ is then second fundamental form of $\partial \Omega$ relative to $\nu$ and $Ric$ is the Ricci curvature.

In particular, if the Ricci curvature of $M$ is bounded from below, then 
$$
\frac{D^2}{d\xi^2}\{h^2(\xi)+(\underset{M}{\text{inf}\;}Ric)(\xi|M|)^2\}\leqslant 0\;\;\;\;\;at\;\;\xi=\frac{|\Omega|}{|M|}.
$$ In this case, the function $\xi \longmapsto h^2(\xi)+(\underset{M}{\text{inf}\;}Ric)(\xi|M|)^2$ is concave, which implies $h(\xi)$ is locally Lipschitz. Also in this case, the function $\xi \longmapsto h^2(\xi)+(\underset{M}{\text{inf}\;}Ric-\epsilon)(\xi|M|)^2$ is strictly concave for any $\epsilon>0$.
\end{prop}

\begin{proof}[\textbf{Proof of Theorem \ref{thm: Uniform Lipschitz cty of h^2}.}]
$\;$\newline
\noindent {\bf Claim 1:} If $K_0:=min\{\underset{(M,g_0)}{\text{inf}\;}K, \text{0}\}$, then $K(p,t)\geq{K_0}$ 
for all points $p\in M$ and all $t \in [0,\infty]$.

\noindent {Proof of Claim 1.}
Under the normalized Ricci flow, we have $\frac{\partial K}{\partial t}=\Delta K+2K(K-\Bar{K})$, which implies that $\frac{\partial }{\partial t}K_{min}\geqslant 2K_{min}(K_{min}-\Bar{K})$. If $\Bar{K}\leqslant{0}$, then $\underset{(M,g_0)}{\text{inf}\;}K \leqslant{\Bar{K}}\leqslant{0}$ and it is easy to see that $\frac{\partial }{\partial t}K_{min}\geqslant{0}$. If $\Bar{K}>0$, we distinguish two cases.\\ 
\noindent {\bf Case 1}: $\underset{(M,g_0)}{\text{inf}\;}K \leqslant{0}$.

We have $K_{min}\big{|}_{t=0}\leqslant{0}$ and thus $\frac{\partial }{\partial t}K_{min}\big{|}_{t=0}\geqslant{0}$. So whenever $K_{min}\leqslant{0}$, we have $\frac{\partial }{\partial t}K_{min}\geqslant{0}$. If $K_{min}$ later becomes strict positive, then $K_{min}$ will always be strict positive. 

\noindent {\bf Case 2}: $\underset{(M,g_0)}{\text{inf}\;}K >{0}$.

We have $K_{min}\big{|}_{t=0}>{0}$, and then $K_{min}$ will always be strict positive. \\Thus claim 1 follows. \newline

\noindent {\bf Claim 2:} For each fixed $t\in[0,+\infty]$, $Q(t,\xi):=h_{g(t)}^{2}(\xi)+K_{0}(\xi|M|_{g(t)})^{2}$ is concave in $\xi \in [0,1]$, where $K_0$ comes from the claim 1.

\noindent {Proof of Claim 2.} Claim 1 and Lemma \ref{lem: Lipschitz cty of h on surfaces} imply for all $t\in[0,+\infty]$, $Q(t,\xi)$ is concave in $\xi \in (0,1)$. Using continuous extension of $h(\xi)$, we see that $Q(t,\xi)$ is left-continuous at $\xi=0$ and right-continuous at $\xi=1$. Hence the super-linear inequality $Q(t,\alpha\xi+\beta\eta)\geqslant \alpha Q(t,\xi) + \beta Q(t,\eta)$ in the definition of concavity is preserved when $\xi=0$ or $\eta=1$ by taking limit. Thus claim 2 follows. \newline 

\noindent {\bf Claim 3:} If $\frac{\partial Q(t,\xi)}{\partial \xi}\big{|}_{\xi=0^+}$ and  $\frac{\partial Q(t,\xi)}{\partial \xi}\big{|}_{\xi=1^-}$ exist, then $$\frac{\partial Q(t,\xi)}{\partial \xi}\big{|}_{\xi=0^+} \leqslant 4\pi|M|_{g_0}\;\;\;\; \text{and}\;\;\;\;\;\frac{\partial Q(t,\xi)}{\partial \xi}\big{|}_{\xi=1^-} \geqslant -4\pi|M|_{g_0}+2K_0|M|^2_{g_0}.$$

\noindent {Proof of Claim 3.} Using Theorem \ref{thm: error coeff control of h(g(t))}, we have
\begin{equation*}
    \begin{split}
\frac{\partial Q(t,\xi)}{\partial \xi}\big{|}_{\xi=0^+}& =\underset{\epsilon\rightarrow 0^+}{\text{lim}\;}\frac{Q(t,\epsilon)-Q(t,0)}{\epsilon-0}
=\underset{\epsilon\rightarrow 0^+}{\text{lim}\;}\frac{h^2_{g(t)}(\epsilon)+K_0|M|^2_{g(t)}\epsilon^2}{\epsilon}
\\ & \leqslant \underset{\epsilon\rightarrow 0^+}{\text{lim}\;}\frac{4\pi|M|_{g(t)}\epsilon+(C_0+K_0|M|^2_{g(t)})\epsilon^2}{\epsilon}
=4\pi|M|_{g_0},
    \end{split}
\end{equation*} where $C_0$ is a uniform constant depending only on $|M|_{g_0}$ and $\text{sup}_{(M,g_0)}K$.

\begin{equation*}
    \begin{split}
\frac{\partial Q(t,\xi)}{\partial \xi}\big{|}_{\xi=1^-} & =\underset{\xi\rightarrow 1^-}{\text{lim}\;}\frac{Q(t,\xi)-Q(t,1)}{\xi-1} =\underset{\xi\rightarrow 1^-}{\text{lim}\;}\frac{h^2_{g(t)}(\xi)}{\xi-1}+2K_0|M|^2_{g(t)} \\ & =\underset{\epsilon\rightarrow 0^+}{\text{lim}\;}\frac{h^2_{g(t)}(1-\epsilon)}{-\epsilon}+2K_0|M|^2_{g(t)}=-\frac{\partial Q(t,\xi)}{\partial \xi}\big{|}_{\xi=0^+}+2K_0|M|^2_{g(t)}
\\ & \geqslant -4\pi|M|_{g_0}+2K_0|M|^2_{g_0}.
    \end{split}
\end{equation*}

\noindent {\bf Claim 4:} $Q(t,\xi)$ (where $\xi\in[0,1]$) is uniform Lipschitz continuous with uniform Lipschitz constant bounded by $4\pi|M|_{g_0}+2|K_0||M|^2_{g_0}$ over time-interval $[0,+\infty]$.

\noindent {Proof of Claim 4.}
By elementary properties of concave function, we have for all $t\in[0,+\infty]$ and for all $0<\eta_1<\eta_2<1$,
\begin{equation*}
    \begin{split}
\frac{\partial Q(t,\xi)}{\partial \xi}\big{|}_{\xi=0^+}
\geqslant 
\frac{\partial Q(t,\xi)}{\partial \xi}\big{|}_{\xi=\eta_1^{-}} & \geqslant 
\frac{\partial Q(t,\xi)}{\partial \xi}\big{|}_{\xi=\eta_1^{+}} \\ &\geqslant 
\frac{Q(t,\eta_1)-Q(t,\eta_2)}{\eta_1-\eta_2} 
\\ & \geqslant 
\frac{\partial Q(t,\xi)}{\partial \xi}\big{|}_{\xi=\eta_2^{-}} \geqslant 
\frac{\partial Q(t,\xi)}{\partial \xi}\big{|}_{\xi=\eta_2^{+}} \geqslant 
\frac{\partial Q(t,\xi)}{\partial \xi}\big{|}_{\xi=1^-}
    \end{split}
\end{equation*}
Hence, $\frac{\partial Q(t,\xi)}{\partial \xi}\big{|}_{\xi=0^+}$ exists but may be $+\infty$ and $\frac{\partial Q(t,\xi)}{\partial \xi}\big{|}_{\xi=1^-}$ exists but may be $-\infty$. But claim 3 tells us that both $\frac{\partial Q(t,\xi)}{\partial \xi}\big{|}_{\xi=0^+}$ and $\frac{\partial Q(t,\xi)}{\partial \xi}\big{|}_{\xi=1^-}$ are finite. 

Let's denote $\text{max}\;\{\big{|}\frac{\partial Q(t,\xi)}{\partial \xi}\big{|}_{\xi=0^+}\big{|}, \big{|}\frac{\partial Q(t,\xi)}{\partial \xi}\big{|}_{\xi=1^-}\big{|}\}$ by $N$, then
$$
\big{|}\frac{Q(t,\eta_1)-Q(t,\eta_2)}{\eta_1-\eta_2}\big{|}\leqslant N \leqslant 4\pi|M|_{g_0}+2|K_0||M|^2_{g_0}\;\;\;\;\forall\eta_1 \neq \eta_2 \in [0,1],
$$ which gives the claim.\newline

Finally, to finish the proof, denote $f(\xi):=-K_0(\xi|M|_{g(t)})^2$ (where $\xi\in[0,1]$), which has uniform Lipschitz constant $2|K_0
||M|^2_{g_0}$ over time-interval $[0,+\infty]$. Then, by claim 4,  $h^2_{g(t)}(\xi)=Q(t,\xi)+f(\xi)$ has uniform Lipschitz constant bounded by $4\pi|M|_{g_0}+2|K_0||M|^2_{g_0}+2|K_0
||M|^2_{g_0}=4\pi|M|_{g_0}+4|K_0||M|^2_{g_0}$ over time-interval $[0,+\infty]$, as desired.
\end{proof}

\begin{proof}[\textbf{Proof of Corollary \ref{cor: Uniform Lipschitz cty of h}}.]
$\;$\newline
For any $t\in [0,T]$ and any $\xi\neq \eta \in [\xi_0, \xi_1]$, we have, denote $K_0=min\{\underset{(M,g_0)}{\text{inf}\;}K, \text{0}\}$,
$$
\big{|}\frac{h_{g(t)}(\xi)-h_{g(t)}(\eta)}{\xi-\eta}\big{|}=\big{|}\frac{1}{h(\xi)+h(\eta)}\big{|}\big{|}\frac{h^2(\xi)-h^2(\eta)}{\xi-\eta}\big{|} \leqslant 2\alpha^{-1}( \pi|M|_{g_0}+|K_0||M|^2_{g_0}),
$$ where $\alpha=\alpha(T):=\text{inf}\;\{h_{g(t)}(\xi):t\in [0,T],\xi \in [\xi_0,\xi_1]\}>0$ by Theorem \ref{thm: joint cty of h}. 
\end{proof}

\begin{rmk}
Because normalized Ricci flow on any compact surface converges to a limiting metric continuously, we have $\alpha(\infty)>0$. To show this precisely, first use the change of variable $s(t)=\frac{t}{1-t}$, for $t\in [0,\infty)$, to get a continuous family of metrics $g(s)$ on $[0,1)$. The convergence of normalized Ricci flow extends this family of metrics $g(s)$ to the interval $[0,1]$ continuously and then $\alpha(\infty)>0$ follows from Theorem \ref{thm: joint cty of h}.
\end{rmk}

\end{document}